\documentclass[reqno, 12pt, reqno]{amsart}

\usepackage{amsfonts, amsthm, amsmath, amssymb}
\usepackage{hyperref}
\hypersetup{colorlinks=false}

\usepackage[margin=3.5cm]{geometry}

\RequirePackage{mathrsfs} \let\mathcal\mathscr

\numberwithin{equation}{section}

\newtheorem{theorem}{Theorem}[section]
\newtheorem{lemma}[theorem]{Lemma}

\theoremstyle{remark}
\newtheorem*{ack}{Acknowledgements}

\newtheorem*{example}{Example}

\theoremstyle{definition}

\newtheorem{definition}[theorem]{Definition}

\renewcommand{\d}{\mathrm{d}}
\renewcommand{\phi}{\varphi}
\renewcommand{\rho}{\varrho}

\newcommand{\0}{\mathbf{0}}
\newcommand{\1}{\mathbf{1}}

\newcommand{\PP}{\mathbb{P}}

\newcommand{\ZZ}{\mathbb{Z}}
\newcommand{\ZZp}{\mathbb{Z}_{\mathrm{prim}}}
\newcommand{\NN}{\mathbb{N}}
\newcommand{\QQ}{\mathbb{Q}}
\newcommand{\RR}{\mathbb{R}}

\renewcommand{\leq}{\leqslant}

\renewcommand{\geq}{\geqslant}

\newcommand{\x}{\mathbf{x}}
\newcommand{\y}{\mathbf{y}}
\renewcommand{\c}{\mathbf{c}}
\renewcommand{\u}{\mathbf{u}}
\newcommand{\z}{\mathbf{z}}

\renewcommand{\b}{\mathbf{b}}

\newcommand{\ve}{\varepsilon}

\newcommand{\bxi}{\boldsymbol{\xi}}

\DeclareMathOperator{\rank}{rank}

\DeclareMathOperator{\Span}{span}

\renewcommand{\t}{\mathbf{t}}

\usepackage[usenames,dvipsnames]{color}

\begin{document}

\title{Free rational points on smooth hypersurfaces}

\author{Tim Browning}

\address{IST Austria\\
Am Campus 1\\
3400 Klosterneuburg\\
Austria}
\email{tdb@ist.ac.at}

\author{Will Sawin}
\address{Columbia University\\ 
Department of Mathematics\\
2990 Broadway\\ New York\\ NY 10027\\ USA}
\email{sawin@math.columbia.edu}


\begin{abstract}
Motivated by a recent question of Peyre, 
we  apply  the  Hardy--Littlewood circle method to count ``sufficiently free" rational points of bounded height on 
 arbitrary smooth projective 
hypersurfaces of   low degree that are defined over the rationals.
\end{abstract}

\date{\today}

\maketitle

\thispagestyle{empty}
\setcounter{tocdepth}{1}
\tableofcontents

\section{Introduction}

Let $V\subset \PP^{n-1}$ be a smooth hypersurface of degree $d\geq 3$, defined over the field of rational numbers. For $B\geq 1$, let $N_V(B)=\#\{x\in V(\QQ): H(x)\leq B\}$, where 
$H$ is the usual exponential height function on 
$\PP^{n-1}(\QQ)$.
Thanks to
the Hardy--Littlewood circle method
and  work of Birch \cite{birch}, it follows that 
there exists a constant $\delta>0$ such that 
\begin{equation}\label{eq:birch}
N_V(B)= c B^{n-d}+O_V(B^{n-d-\delta}),
\end{equation}
as $B\to \infty$, provided that $n>2^d(d-1)$.
Here $c=\frac{1}{n-d}\omega_H(V(\mathbf{A}_\QQ))$ and
$\omega_H$ is the Tamagawa measure on the space of adeles of $V$. The asymptotic formula \eqref{eq:birch} provided one of the earliest pieces of evidence for the conjecture of Manin 
\cite{FMT89}, and its refinement by  Peyre \cite{peyre-duke}, about the distribution of rational points on Fano varieties. 

The purpose of this paper is to address a very recent question of  Peyre \cite{peyre-freedom} about 
the distribution of ``sufficiently free'' rational points of bounded height on $V$. 
Peyre associates a measure of ``freeness'' $\ell(x)\in [0,1]$ to any $x\in V(\QQ)$ and advocates the idea of only counting those rational points which satisfy $\ell(x)\geq \ve_B$, where $\ve_B$ is a function of $B$ decreasing to zero sufficiently slowly.\footnote{A similar question was asked by Ellenberg and Venkatesh in a 2015 private communication with the first author.} 
(See \cite[Def.~6.11]{peyre-freedom} for a precise statement for arbitrary Fano varieties over arbitrary number fields.) Peyre's function $\ell(x)$ is defined in \eqref{eq:f-P} using 
Arakelov geometry and 
the  theory of slopes associated to the tangent bundle $\mathcal{T}_V$.  
Let 
\begin{equation}\label{eq:NV}
N_V^{\ve\text{-free}}(B)=\#\left\{x\in V(\QQ): \ell(x)\geq \ve, ~H(x)\leq B\right\}.
\end{equation}
In the setting of smooth hypersurfaces $V\subset \PP^{n-1}$ of low degree, 
Peyre predicts that for a suitable range of $\ve$, 
$N_V^{\ve\text{-free}}(B)$ should have the same asymptotic behaviour as the usual counting function 
$N_V(B)$, as $B\to \infty$. 
The following result 
confirms this for a  range of $\ve$ that is independent of $B$.

\begin{theorem}\label{t:P}
Let $d\geq 3$ and let $n>3(d-1)2^{d-1}$. Then there exists a constant $c_{d,n}\in (0,1)$ 
such that for any 
$$
0\leq \ve < c_{d,n}
$$
there exists a further  constant $\delta>0$
such that  
$$
N_V^{\ve\text{-free}}(B)=c
 B^{n-d} +O_{V,\ve}\left(B^{n-d-\delta} \right),
$$
where $c=\frac{1}{n-d}\omega_H(V(\mathbf{A}_\QQ))$ is the expected  leading constant.
\end{theorem}

Note that in our theorem the parameter $\ve$ is a constant, while in Peyre's notion of freeness one takes $\ve_B$ tending to zero. Thus our result is stronger than necessary for Peyre's formulation.
We shall show in \S \ref{s:peyre} that it suffices to work with a simpler freeness function 
$\widetilde\ell(x) $ that is defined in \eqref{eq:free} in terms of the largest successive minimum of a certain   associated lattice. Once this is achieved, the proof of Theorem \ref{t:P}
is guided by our investigation \cite{BS2} of the analogous situation for smooth hypersurfaces over  global fields of positive characteristic.  We shall find that the role of the Riemann--Roch theorem in \cite[\S 3]{BS2} is replaced by the Poisson summation formula.
After this the  argument runs in close parallel to \cite{BS2}, apart from in one essential difference associated to 
primes of bad reduction for $V$.

An interesting feature of our method is that it relies on counting integer solutions $(\x,\y)$ to the system of equations $f(\x) = \y.\nabla f(\x)=0$,  where $f$ is the defining polynomial of $V$. This is equivalent to counting integer points on the tangent bundle of the affine cone over $V$. This  suggests that it may be possible to bound the number of rational points of small freeness on a Fano variety $X$ by using asymptotics for the number of rational points on $X$ together with asymptotics for the  number of integral points on the tangent bundle of $X$.

\begin{ack}
The authors are very grateful to the anonymous referee for numerous pertinent remarks. 
While working on this paper the first  author was
supported by EPRSC 
grant \texttt{EP/P026710/1}. 
The research was partially conducted during the period the second author served as a Clay Research Fellow, and partially conducted during the period he was supported by Dr.\ Max R\"{o}ssler, the Walter Haefner Foundation and the ETH Zurich Foundation.
\end{ack}

\section{The geometry of numbers and the shape of lattices}

Most of the facts that we record in this section are taken from the book by Cassels \cite{cassels}. Recall that a {\em lattice} $\Lambda$ is a discrete additive subgroup of $\RR^n$.  Equivalently 
$$
\Lambda=\{x_1\mathbf{b}_1+\dots
+x_r\mathbf{b}_r: x_i\in \ZZ
\},
$$
for a set of linearly independent vectors $\b_1,\dots,\b_r\in \RR^n$. The {\em rank} of $\Lambda$ is then $\rank(\Lambda)=r$ and the {\em determinant} is $\det(\Lambda)=\sqrt{\det(B^tB)}$, where
$B$ is the $n\times r$ matrix formed from the column vectors $\b_1,\dots,\b_r$.
For each $1\leq k\leq r$ let $s_k(\Lambda)$ be the least $\sigma>0$ such that 
$\Lambda$ contains at least $k$ linearly independent vectors of Euclidean length bounded by $\sigma$. 
The $s_k(\Lambda)$ are the {\em successive minimima} of $\Lambda$ and they satisfy 
$
0<s_1(\Lambda)\leq s_2(\Lambda)\leq \dots \leq s_r(\Lambda).
$
Furthermore, it follows from Minkowski's second convex body theorem
\cite[\S VIII.3.2]{cassels}
 that
\begin{equation}\label{eq:upper-lower}
\det(\Lambda) \leq \prod_{i=1}^r s_i (\Lambda)\ll_n \det(\Lambda), 
\end{equation}
where the implied constant depends only on $n$.
The {\em dual lattice} is defined to be 
$$
\Lambda^*=\{\x\in \Span_\RR(\Lambda): \x.\y\in \ZZ \text{ for all $\y\in \Lambda$}\}
$$
This lattice has basis matrix $B(B^tB)^{-1}$ and so 
$\rank(\Lambda^*)=r$ and
$\det(\Lambda^*)=\det(\Lambda)^{-1}$. Appealing to work of 
Banaszczyk \cite[Thm.~2.1]{ban}, it  follows that 
\begin{equation}\label{eq:*}
1\leq s_k(\Lambda) s_{r-k+1}(\Lambda^*) \leq r ,
\end{equation}
for $1\leq k\leq r$.

The following result is well-known and will prove instrumental in our work. A proof is given as a special case of work by Heath-Brown \cite[Lemma 1]{hb84}.

\begin{lemma}\label{lem:lattice1}
For any  vector $\c\in\ZZp^{n}$ the set
$\Lambda=\{\x\in\ZZ^{n}: \c.\x=0\}$ is a lattice of dimension
$n-1$ and determinant $\det(\Lambda)=\|\c\|$, where $\|\cdot\|$ is the Euclidean norm on $\RR^n$.
\end{lemma}

Given a lattice $\Lambda\subset \RR^n$ of rank $r$ it will be important to detect when the lattice is unusually skew, in the sense that the largest successive minimum is excessively large.  To be precise, we 
 seek a useful majorant for the indicator function
$$
\mathbf{1}_R(\Lambda)=\begin{cases}
1 & \text{ if $s_r(\Lambda)>R$,}\\
0& \text{ otherwise}.
\end{cases}
$$
This is achieved in the following simple result. 
 
\begin{lemma}\label{lem:gauss}
Let $\Lambda\subset \RR^n$ be a lattice of rank $r\leq n$ and let 
 $\omega:\RR^n\to \RR$ be the Gaussian function
$
\omega(\t)=\exp(-\pi \|\t\|^2).
$
Then
$$
\mathbf{1}_R(\Lambda)\leq \exp(\pi r^2)\frac{\det(\Lambda)}{R^r}\left(
\sum_{\y \in \Lambda} \omega(\y/R) - \frac{R^r}{\det(\Lambda)}\right).
$$
\end{lemma}

\begin{proof}
Note that $\omega(\t)\geq 0$ for all $\t\in \RR^n$ and 
$\omega(\0)=1$.
It follows from Poisson summation that 
\begin{align*}
\sum_{\y \in \Lambda} \omega(\y/R) 
&=
 \frac{R^r}{\det(\Lambda)} 
\sum_{\y \in \Lambda^*} \omega(R\y)  
=
 \frac{R^r}{\det(\Lambda)} 
\left(1+\sum_{\substack{\y \in \Lambda^*\\ \y\neq \0}} \omega(R\y) \right),
 \end{align*}
 since $\widehat \omega =\omega$.
  Thus 
$$
\sum_{\y \in \Lambda} \omega(\y/R) - \frac{R^r}{\det(\Lambda)}\geq 0
$$ 
for any lattice $\Lambda$.
Moreover, 
according to \eqref{eq:*}, 
we have 
$
s_1(\Lambda^*)<r/R
$
if $s_r(\Lambda)>R$.
This means that there exists a non-zero vector $\y_0\in \Lambda^*$ such that 
$\|\y_0\|<r/R$. But then 
$$
\sum_{\substack{\y \in \Lambda^*\\ \y\neq \0}} \omega(R\y)\geq 
\omega(R\y_0)
=\exp(-\pi R^2\|\y_0\|^2)
\geq \exp(-\pi r^2).
$$
This implies that 
$$
\sum_{\y \in \Lambda} \omega(\y/R) - \frac{R^r}{\det(\Lambda)}\geq 
\frac{\exp(-\pi r^2) R^r}{\det(\Lambda)}
$$
if $s_r(\Lambda)>R$, which thereby completes the proof of the lemma. 
\end{proof}

\section{Free rational points on hypersurfaces}\label{s:peyre}

Suppose that $f\in \ZZ[x_1,\dots,x_n]$ is a non-singular form of degree $d$ that 
defines the hypersurface $V\subset \PP^{n-1}$. Any rational point $x\in V(\QQ)$ has a representative vector $\x\in \ZZp^n$ such that 
$f(\x)=0$ and $x=(x_1:\dots:x_n)$.  The measure of freeness of $x$ that we shall use in our paper is phrased in terms of the ``well-shapedness'' of the associated lattice
$$
\Lambda_x =\{\y\in \ZZ^n: \y.\nabla f(\x)=0\}.
$$
It follows from Lemma \ref{lem:lattice1} that $\Lambda_x \subset \ZZ^n$ is a lattice of rank $n-1$ and determinant
$$
\det(\Lambda_x)= \frac{\|\nabla f(\x)\|}{\gcd(\nabla f (\x))},
$$
where $\|\cdot\|$ is the Euclidean norm.
Let
$\Delta_f\neq 0$ be the absolute value of the discriminant of the non-singular polynomial $f$.
 From the definition of the discriminant as the resultant of the forms $\partial f/\partial x_1, \dots ,
\partial f/\partial x_n$, it follows that there 
exists $e\in \NN$ and   algebraic identities
 \begin{equation}\label{eq:elimination}
\Delta_f x_i^e =\sum_{1\leq j\leq n} g_{i,j}(\x)\frac{\partial f}{\partial x_i}(\x),
\end{equation}
for $1\leq i\leq n$, where each $g_{i,j}$ has integer coefficients.
 In particular 
 \begin{equation}\label{eq:elimination'}
 \gcd(\nabla f(\x)) \mid \Delta_f
 \quad \text{ for all $\x\in \ZZp^n$}.
 \end{equation}
 
 Next we claim that 
\begin{equation}\label{eq:grad}
\|\x\|^{d-1} \ll
\|\nabla f(\x)\|\ll \|\x\|^{d-1} \quad \text{ for all $\x\in \RR^n$},
\end{equation}
for appropriate implied constants that depend only on $f$. 
Since $f$ has degree $d$ and so its partial derivatives have degree $d-1$, the 
upper bound is  clear.
To see the lower bound we note that  $\nabla f(\x)\neq \0$ for all $\x\in \RR^n$, 
since $f$ is non-singular. 
 Thus 
$\|\nabla f(\x)\|$
 is nowhere vanishing on the unit sphere and so   attains some minimum value $C$, say, there. Thus we have 
 $\| \nabla f(\x) \| \geq C \|\x\|^{d-1}$ in general because $f$ is homogeneous of degree $d$.
This establishes \eqref{eq:grad}.
 
 As we shall see shortly, 
Peyre defines a freeness function relative to the smallest slope on the tangent bundle $\mathcal{T}_V$. 
The measure of freeness that we shall work with is related to this, but it is phrased in terms of the relative size of the largest successive minimum of the lattice $\Lambda_x$. To be precise, we set
\begin{equation}\label{eq:free}
\widetilde{\ell}(x)=\frac{\log \|\x\| -\log s_{n-1}(\Lambda_x)}{\log \|\x\|}.
\end{equation}
Then   $\widetilde{\ell}(x)\geq \ve$ if and only if $s_{n-1}(\Lambda_x)\leq \|\x\|^{1-\varepsilon}$.

We gain some feeling for the behaviour of $\widetilde{\ell}(x)$ by recalling 
\eqref{eq:upper-lower}. Thus for ``typical'' $x$ one might expect 
the successive minima $s_k(\Lambda_x)$ to have the same order of magnitude, for $1\leq k\leq n-1$. If this were true it would follow 
from 
\eqref{eq:upper-lower} 
that 
$$
s_{n-1}(\Lambda_x) \asymp \left(
s_1(\Lambda_x)\dots 
s_{n-1}(\Lambda_x)
\right)^{1/(n-1)} \asymp
\det (\Lambda_x)^{1/(n-1)} \asymp
\|\x\|^{1-\frac{n-d}{n-1}},
$$
since   
$\|\nabla f(\x)\|\asymp \|\x\|^{d-1}$
by \eqref{eq:grad}.
  Such $x$ satisfy $\widetilde{\ell}(x)= (n-d)/(n-1) +o(1)$, as $H(x)\to \infty$.
The following example shows a familiar situation in which the freeness function is unusually small.

\begin{example}
Consider the case $d=3$ and $n=4$ of a smooth cubic surface $V\subset \PP^3$.
Let $L\subset V$ be a $\QQ$-line and define the associated rank $2$ lattice
$$
\mathsf L =\{\0\}\cup \{\z\in \ZZ^4: (z_1:\dots:z_4)\in L\} \subset\ZZ^4.
$$
We claim that, for any $\ve>0$, we have  $\widetilde{\ell}(x)<\varepsilon$ for all but finitely many  $x\in L$.
To see this we note that $f(\x+t\z)$ vanishes identically in $t$ for all $z\in L$. 
But then it follows that 
$
\mathsf L\subset \Lambda_x, 
$
in which case we have $1\leq s_1(\Lambda_x)\leq s_2(\Lambda_x)\ll_L 1$. It now follows from 
\eqref{eq:upper-lower} and \eqref{eq:grad} that 
$$
s_3(\Lambda_x)\gg_V  \frac{\|\nabla f(\x)\|}{s_1(\Lambda_x)s_2(\Lambda_x)}\gg_{L,V} \|\x\|^2.
$$
This therefore yields $\widetilde{l}(x) \leq -1 +o(1)$ and the claim.
\end{example}

We now explain how our freeness function \eqref{eq:free} relates to that defined by Peyre
\cite[D\'ef.~4.11]{peyre-freedom}. 
To begin with we can extend $V$ to a closed subscheme $V \subset \mathbb P^{n-1}_{\mathbb Z}$. A rational point $x \in V(\mathbb Q)$  gives a section $x \in V(\mathbb Z)$ of this scheme. Because this scheme is smooth of dimension $n-2$, the pullback
$ (\mathcal T_V)_{x}$ of its tangent bundle along ${x}$ is a rank $n-2$ free $\mathbb Z$-module; i.e. a free lattice of rank $n-2$. Fixing a Riemannian metric on $V(\mathbb R)$ gives a metric on this lattice. Peyre defines the freeness of $x$ as 
\begin{equation}\label{eq:f-P}
\ell(x) =\frac{   \max\left\{(n-2)  \mu_{n-2} ( (\mathcal T_V)_{x}), 0\right\} }{  h(x) } ,\end{equation}
 where $h(x)=(n-d)\log \|\x\| +O(1)$ is the logarithmic anticanonical height of $x$ and 
$\mu_1 \geq \dots \geq \mu_{n-2}$ are the slopes defined by Bost. 
There are four main differences between Peyre's definition and ours:
\begin{enumerate}

\item Peyre includes a factor of 
$n-2$ in the numerator  
and the anticanonical height in the denominator instead of $\log \|\x\|$. 

\item Peyre uses the notion of slopes instead of successive minima. The slopes of a lattice differ from minus the logarithms of its successive minima by $O(1)$.
 
\item Peyre works in a slightly different lattice, namely the tangent lattice instead of the perpendicular lattice to $\nabla f(\x)$. These lattices are closely related, but not identical, and this discrepancy means that we only  produce 
an inequality (instead of an identity) between the two notions of freeness. 

\item Peyre defines the freeness to always be non-negative.
\end{enumerate}
The relationship between the two notions of freeness is articulated in the following result.

\begin{lemma}\label{lem:Peyre}
For any $x\in V(\QQ)$ we  have 
$$
\ell(x) \geq   \frac{n-2}{n-d} \widetilde{\ell}(x) + O \left(\frac{1}{h(x)}\right) .
$$  
\end{lemma}

\begin{proof}
We first explain how to relate the tangent lattice to $\Lambda_{x}$, and then why this leads to the stated  inequality. 
We have an Euler exact sequence  
\[ 0 \to \mathcal O_{\mathbb P^{n-1}} \to  \mathcal O_{\mathbb P^{n-1}} (1) ^n \to \mathcal T_{\mathbb P^{n-1} } \to 0\] 
on $\PP_\ZZ^{n-1}$.
This induces an exact sequence \[ 0 \to \left( \mathcal O_{\mathbb P^{n-1}}\right)_{x} \to \left( \mathcal O_{\mathbb P^{n-1}} (1) ^n\right)_{x}  \to \left(\mathcal T_{\mathbb P^{n-1} } \right)_{x} \to 0.\] We have
$\left( \mathcal O_{\mathbb P^{n-1}}\right)_{x} = \mathbb Z$ and  $\left(  \mathcal O_{\mathbb P^{n-1}} (1) ^n\right)_{x}  = \mathbb Z^n$ because $\mathcal O_{\mathbb P^{n-1}}$ and $\mathcal O_{\mathbb P^{n-1}} (1) $ are rank one locally free sheaves, so their pullback along $x$ are rank one locally free sheaves on $\operatorname{Spec} \mathbb Z$, which are all isomorphic to $\mathbb Z$. The map 
between them is multiplication by $\x$, so the tangent lattice of $\mathbb P^{n-1}$ is the quotient lattice $\mathbb Z^n / \mathbb Z \x$. We claim that 
the induced metric on this is 
the renormalized metric 
\[ \| \y+ \mathbb Z \x \| = \frac{ \min_{ t\in \mathbb R} ( \| \y - t \x\|) }{ \| \x \| }.\] Formally this arises from the
 $\mathcal O_{\PP^{n-1}}(1)$ twist, but we can see this explicitly since the natural isomorphism between $\mathbb R^n/\mathbb R\x$ and the tangent space to $\mathbb P^{n-1}_{\mathbb R}$ at $\x$ depends on the scaling of the vector $\x$ and not just on its equivalence class in $\mathbb P^{n-1}$. The Arakelov metric on the tangent bundle of projective space must depend continuously on a point in projective space. To make it do so, we divide by $\| \x \| $.

Calculating $\left( \mathcal T_{V}\right)_{x}$ is now relatively easy. Consider  the exact sequence \[ 0 \to \mathcal T_{V} \to \mathcal T_{\mathbb P^{n-1}} \to \mathcal O_V (d) \to 0,\] where the second map represents dot product with $\nabla f(\x)$. We can realise  $\left( \mathcal T_{V}\right)_{x}$ as the kernel of dotting with $\nabla f (\x)$ in $\mathbb Z^n/\mathbb Z \x$, with no further renormalization necessary.  Invoking some basic properties of slopes, we deduce that
\begin{align*}
\mu_{n-2} ( (\mathcal T_V)_{x}) &= \mu_{n-2} ( \Lambda_x / \mathbb Z\x) + \log \|\x \|  \\
&\geq \mu_{n-1} ( \Lambda_x) + \log \| \x\|  \\
&\geq 
-s_{n-1} ( \Lambda_{x} )   + \log \| \x\|.
\end{align*} Indeed, the first step uses the fact that, when we divide the metric of a lattice by $\|\x\|$, we add $\log \| x\|$ to each slope of the lattice, which is clear from the definition \cite[D\'ef~4.4]{peyre-freedom} and is a special case of \cite[Lemma 4.2]{Bost}. The second step uses the fact that the minimum slope of a quotient lattice is at least the minimum slope of the original lattice, which is immediate from the definition of the minimum slope as a minimum over quotients of the lattice in \cite[p.~195]{Bost} and the equivalence of Bost's minimum slope and the last slope in Peyre's ordering. The last step uses \cite[Remarque~4.7(b)]{peyre-freedom}. 

The inequality 
\[ \frac{ (n-2)  \mu_{n-2} ( (\mathcal T_V)_{x})}{  h(x) } \geq \frac{n-2}{n-d} \widetilde{\ell}(x) + O \left(\frac{1}{h(x)}\right) \] immediately follows, since
 $h(x) = (n-d) \log  \|\x\|+O(1)$.  We therefore 
 have 
 \begin{align*}
 l(x) = \frac{   \max\left\{(n-2)  \mu_{n-2} ( (\mathcal T_V)_{x}), 0\right\} }{  h(x) } &\geq  \frac{(n-2)  \mu_{n-2} ( (\mathcal T_V)_{x}) }{  h(x) } \\
 &\geq \frac{n-2}{n-d} \widetilde{\ell}(x) + O \left(\frac{1}{h(x)}\right) 
 \end{align*}
  except if $h(x)<0$ where the middle inequality fails, but this happens for only finitely many $x$ and we can handle it by assuming that the constant in the $ O (1/h(x))$ term is sufficiently large.
\end{proof}

Returning to
\eqref{eq:NV}, we can now make sense of the counting function
\begin{align*}
N_V^{\ve\text{-free}}(B)
&=\#\left\{x\in V(\QQ): \ell(x)\geq \ve, ~H(x)\leq B\right\}\\
&=\#\left\{x\in V(\QQ): H(x)\leq B\right\}-E_{V,\ve}(B),
\end{align*}
for any $\ve>0$, 
where
$$
E_{V,\ve}(B)=\#\left\{x\in V(\QQ): \ell(x)<\ve, ~H(x)\leq B\right\}.
$$
The first term is handled by   \eqref{eq:birch}, since $3(d-1)2^{d-1}\geq 2^d(d-1).$
Moreover, in view of Lemma \ref{lem:Peyre} and the fact that $d\geq 3$,  we have 
\begin{align*}
E_{V,\ve}(B)
&\leq \#\left\{x\in V(\QQ): \widetilde \ell(x)< \frac{n-d}{n-2}\ve +o(1), ~H(x)\leq B\right\}+O(1)\\
&\leq \#\left\{x\in V(\QQ): \widetilde \ell(x)< \ve, ~H(x)\leq B\right\} + O(1) ,
\end{align*} 
where the presence of the $O(1)$ term is needed to account for the low height points.
Hence 
\begin{align*}
E_{V,\ve}(B)
&\leq \frac{1}{2}\#\left\{\x\in \ZZp^n: 
\begin{array}{l}
f(\x)=0,~\|\x\|\leq B\\
s_{n-1}(\Lambda_x)>\|\x\|^{1-\ve}
\end{array}
\right\} + O(1) ,
\end{align*}
on taking into account the action of the units $\{\pm 1\}$  on $\PP^{n-1}(\QQ)$.
We require an upper bound for $E_{V,\ve}(B)$ which is $O_{V,\ve}(B^{n-d-\delta})$
for an appropriate $\delta>0$, and which is valid for as wide a range of $\ve$ as possible. 

To handle $E_{V,\ve}(B)$ it will be convenient to break the range for $\|\x\|$ into dyadic intervals. Thus
\begin{equation}\label{eq:B-R}
\begin{split}
E_{V,\ve}(B)
&\leq 
\sum_{\substack{R=2^j\\1\leq R\leq 2 B}}
\#\left\{\x\in \ZZp^n: 
\begin{array}{l}
f(\x)=0,~R/2<\|\x\|\leq R\\
s_{n-1}(\Lambda_x)>R^{1-\ve}
\end{array}
\right\}+O(1)\\
&=\sum_{\substack{R=2^j\\1\leq R\leq 2B}}
E_{V,\ve}^*(R)+O(1),
\end{split}\end{equation}
say.
Appealing to Lemma \ref{lem:gauss}, we deduce that 
\begin{equation}\label{eq:E_eps}
E_{V,\ve}^*(R) \ll1+
\sum_{\substack{
\x\in \ZZp^n\\
f(\x)=0\\R/2<\|\x\|\leq R}}
\frac{\det(\Lambda_x)}{R^{(1-\ve)(n-1)}}\left(
\sum_{\y \in \Lambda_x} \omega(\y/R^{1-\ve}) - \frac{R^{(1-\ve)(n-1)}}{\det(\Lambda_x)}\right),
\end{equation}
where  $\omega(\t)=\exp(-\pi\|\t\|^2)$.
In what follows it will be convenient to write
$\mathsf d(\x)=\det(\Lambda_x)$ 
when $x$ is represented by a vector $\x\in \ZZp^n$.
In view of \eqref{eq:elimination'} 
we have $\gcd(\nabla f (\x))=\gcd(\nabla f (\x),\Delta_f)$, so that 
\begin{equation}\label{eq:dx}
\mathsf d(\x)=
 \frac{\|\nabla f(\x)\|}{\gcd(\nabla f (\x),\Delta_f)}.
\end{equation}
We shall use this formula to extend the definition of 
$\mathsf d(\x)$ to all (not necessarily primitive) vectors $\x\in \ZZ^n$.

Let us write $\|\x\|\sim R$ to denote the inequalities $R/2<\|\x\|\leq R$.
In order to treat 
$E_{V,\ve}^*(R)$ we begin by analysing the 
term 
$$
M_\ve(R)
=
\sum_{\substack{
\x\in \ZZp^n\\
f(\x)=0\\\|\x\|\sim  R}}
\mathsf{d}(\x)
\sum_{\y \in \Lambda_x} \omega(\y/R^{1-\ve}).
$$
It is clear that   $\y.\nabla f(k\x)=0$ if and only if $\y.\nabla f(\x)=0$, 
for any $k\in\NN$.
Hence, an application of M\"obius inversion 
yields
\begin{equation}\label{eq:MeR}
\begin{split}
M_\ve(R)
&=
\sum_{k\leq R}\mu(k)
\sum_{\substack{
\x\in \ZZ^n\\
f(\x)=0\\\|\x\|\sim R/k}}
\mathsf{d}(k\x)
\sum_{\substack{\y \in \ZZ^n\\ 
\y.\nabla f(\x)=0} } \omega(\y/R^{1-\ve})\\
&=
\sum_{k\leq R}\mu(k)
\sum_{\substack{
\x\in \ZZ^n\\
f(\x)=0\\\|\x\|\sim R/k}}
\mathsf{d}(k\x)
\int_0^1 S(\beta) \d\beta,
\end{split}
\end{equation}
where
$$
S(\beta)=
 \sum_{\y \in \ZZ^n} \omega(\y/R^{1-\ve})  e(\beta \y.\nabla f(\x)) .
$$
Our plan is to define a set of ``major arcs'' for the interval $[0,1]$ whose integral matches the expected main term 
$R^{(1-\ve)(n-1)}/\det(\Lambda_x)$ from 
\eqref{eq:E_eps}.

\section{Identification of the major arcs}

Our identification of the major arcs follows the path that was paved in  \cite[\S 4]{BS2}.
Henceforth all implied constants will be allowed to depend on $f$.
It will be convenient to set
$$
X=X_k = R/k \quad \text{ and } \quad Y=R^{1-\ve},
$$
where $k$ is the parameter occurring in \eqref{eq:MeR}.
Since $\hat \omega=\omega$, it follows from  Poisson summation that 
\begin{align*}
S(\beta)
&=
 \sum_{\y \in \ZZ^n} \omega(\y/Y)  e(\beta \y.\nabla f(\x)) \\
 &=Y^n\sum_{\y\in \ZZ^n} \omega(Y\beta \nabla f(\x)-Y\y), 
\end{align*}
 for any $\beta\in [0,1]$ and any  $\x\in \ZZ^n$. 
Let us use $\langle \alpha\rangle=\inf_{m\in \ZZ}|m-\alpha|$ to denote the distance to the nearest integer. 
We observe that 
$$
\omega(\t)\ll_N (1+\|\t\|)^{-N},
$$
for any  $\t\in \RR^n$ 
and any $N\geq 0$.
Hence it is not hard to  see that 
\begin{equation}\label{eq:Sb-1}
S(\beta)=Y^n
\omega\left(Y\left\langle\beta\frac{\partial f(\x)}{\partial x_1}\right\rangle,
\dots, 
Y\left\langle\beta\frac{\partial f(\x)}{\partial x_n}\right\rangle
\right)
+O_N(Y^{-N}),
\end{equation}
for any $N\geq 0$.
Led by this   we make the following definition.

\begin{definition}[Major arcs]\label{def:major}
For any $\eta>0$ we
 set
$$
\mathfrak{M}_\eta(X,Y)=\bigcup_{q\leq Y^{1-\eta}} \bigcup_{\substack{0\leq a<q\\ \gcd(a,q)=1}} \left\{\beta\in [0,1): \left|q\beta-a\right|\leq \frac{1}{C_fX^{d-1}}\right\},
$$
where $C_f>0$ is a sufficiently large constant that only depends on $f$.
\end{definition}

The following result is concerned with the size of the exponential sum $S(\beta)$ when $\beta$
belongs to this set of major arcs.

\begin{lemma}\label{lem:Sbeta}
Let $N\geq 0$, let 
$\x\in \ZZ^n$ with
$\|\x\|\leq X$,  and let $\beta =a/q+\theta\in \mathfrak{M}_\eta(X,Y)$ for coprime integers $a,q$ such that 
$0\leq a<q$ and $|q\theta| \leq 1/ (C_f X^{d-1})$. Then 
$$
S(\beta)=
Y^n\omega(Y\theta \nabla f(\x))+
O_N(Y^{-N}) 
$$
if $q\mid \nabla f(\x)$ and $|\theta|\leq Y^{-1+\eta}/\|\nabla f(\x)\|$, with 
$S(\beta)=O_N(Y^{-N})$ otherwise.
\end{lemma}

\begin{proof}
Let $\beta=a/q+\theta\in \mathfrak{M}_\eta(X,Y)$. Then
$$
\left|\theta\frac{\partial f(\x)}{\partial x_i}\right|
\leq \frac{1}{4q}
$$
for any $1\leq i\leq n$, provided that $C_f$ is large enough.
Next, we see that
$$
\left\langle\beta\frac{\partial f(\x)}{\partial x_i}\right\rangle=
\left|\theta\frac{\partial f(\x)}{\partial x_i}\right|
$$
if $q\mid \partial f(\x)/\partial x_i$. Thus 
it follows from 
 \eqref{eq:Sb-1}
 that $S(\beta)=O_N(Y^{-N})$ 
 if $q\mid \nabla f(\x)$ and $|\theta|> Y^{-1+\eta}/\|\nabla f(\x)\|$.
 Alternatively, 
 if we have $q\mid \nabla f(\x)$ and $|\theta|\leq  Y^{-1+\eta}/\|\nabla f(\x)\|$ then 
 clearly
$$
 Y^n
 \omega\left(Y\left\langle\beta\frac{\partial f(\x)}{\partial x_1}\right\rangle,
\dots, 
Y\left\langle\beta\frac{\partial f(\x)}{\partial x_n}\right\rangle
\right)=
Y^n\omega(Y\theta \nabla f(\x)).
$$  
Finally, 
if $q\nmid \partial f(\x)/\partial x_i$ for some $i\in \{1,\dots,n\}$   then there exists a non-zero integer $u\in [-q/2,q/2]$ such that 
$a \partial f(\x)/\partial x_i\equiv u \bmod{q}$, whence
$$
\left\langle\beta\frac{\partial f(\x)}{\partial x_i}\right\rangle\geq \frac{3}{4q}\geq \frac{3}{4Y^{1-\eta}}, 
$$
for $\beta\in \mathfrak{M}_\eta(X,Y)$. 
This shows that $S(\beta)=O_N(Y^{-N})$ in this case, as required to complete the proof of the lemma.
\end{proof}

The following result is concerned with the evaluation of the  integral of $S(\beta)$ over the  major arcs.

\begin{lemma}\label{lem:major'}
Let $N\geq 0$ and  assume that $\|\x\|\sim X$. Then 
$$
\int_{\mathfrak{M}_\eta(X,Y)} S(\beta)\d\beta=
\frac{Y^{n-1}\gcd(\nabla f(\x))}{\|\nabla f(\x)\|}\left(1+
O\left( \1(\x)\right)\right)
+
O_N(Y^{-N}),
$$
where
$$\1(\x)
=
\begin{cases}
1 & \text{
if ${ \gcd (\nabla f (\x) ) C_f^2> Y^{ 1-  \eta}}$},\\
0 & \text{ otherwise.}
\end{cases}
$$
\end{lemma}
\begin{proof}
Let us set $h=
\gcd(\nabla f(\x))$ throughout the proof.
We  define the modified  major arcs $\widetilde{\mathfrak{M}_\eta}(X,Y)$ to be the set
of $\beta=a/q+\theta\in \mathfrak{M}_\eta(X,Y)$ for which 
$q\mid h$ and $|\theta|\leq Y^{-1+\eta}/\|\nabla f(\x)\|$.
We claim that these modified major arcs are non-overlapping. To see this we suppose that 
$a_1/q_1+\theta_1=a_2/q_2+\theta_2$. Then we may assume without loss of generality that 
$q_1=q_2=h$.  But then it follows that 
$$
|a_1-a_2|= h |\theta_2-\theta_1|\ll X^{d-1} \cdot \frac{1}{C_f X^{d-1}}.
$$
Assuming that $C_f$ is sufficiently large, this implies that $a_1=a_2$, which thereby establishes the claim. (In fact it is not hard to check that the major arcs
$\mathfrak{M}_\eta(X,Y)$ are also disjoint provided that $X^{d-1}\gg Y^{1-\eta}$.)

An application of Lemma~\ref{lem:Sbeta} yields
\begin{align*}
\int_{{\mathfrak{M}_\eta}(X,Y)} S(\beta)\d\beta
&=
\int_{\widetilde{\mathfrak{M}_\eta}(X,Y)} S(\beta)\d\beta+O_N(Y^{-N})\\
&=
Y^n
\sum_{\substack{q\leq Y^{1-\eta}\\ q\mid h }}
\phi(q)
\int_{|\theta|\leq \Theta}
\omega(Y\theta \nabla f(\x)) \d\theta+
O_N(Y^{-N}),
\end{align*}
for any $N\geq 0$,
where
$$
\Theta=\min\left\{
\frac{Y^{-1+\eta}}{\|\nabla f(\x)\|}, ~\frac{1}{qC_f X^{d-1}}\right\}.
$$
Since $(1*\phi)(h)=h$, 
it is clear that 
$$
\sum_{\substack{q\leq Y^{1-\eta}\\ q\mid h }}
\phi(q)=h+O\left(h \1(\x)\right),
$$
where $\1(\x)$ is as in the statement of the lemma.

Next, 
we observe that 
$$
\int_{|\theta|\leq 
Y^{-1+\eta}/\|\nabla f(\x)\|}
\omega(Y\theta \nabla f(\x)) \d\theta=
\frac{1}{Y\|\nabla f(\x)\|} +O_N(Y^{-N}).
$$
Moroever, 
$$
\int_{ \Theta <|\theta|\leq  Y^{-1+\eta}/\|\nabla f(\x)\|}
\omega(Y\theta \nabla f(\x)) \d \theta 
$$
vanishes unless $\Theta <\|\nabla f(\x)\|$, which implies that 
$$
\frac{1}{q C_f X^{d-1} }<  \frac{Y^{-1+\eta}}{\|\nabla f(\x)\|}.
$$
Appealing to 
\eqref{eq:grad} and using the fact that $\|\x\|\sim X$, the right hand side is at most 
$C_f Y^{-1+\eta}/X^{d-1}$, if the constant $C_f$ is taken to be sufficiently large in 
Definition~\ref{def:major}. Hence we conclude that 
$q  C_f^2 >   Y^{ 1-  \eta}    $, which in turn implies that $ h C_f^2> Y^{ 1-  \eta}$ and thus that $\1(\x)=1$ . Because the integrand is nonnegative, the integral over this restricted interval is at most $\frac{1}{Y\|\nabla f(\x)\|}$, so that
\begin{align*}
 \int_{ |\theta| \leq \Theta}  \omega(Y\theta \nabla f(\x)) \d\theta
 =~&
  \int_{|\theta|\leq 
Y^{-1+\eta}/\|\nabla f(\x)\|}
\omega(Y\theta \nabla f(\x)) \d\theta \\
&
- \int_{ \Theta <|\theta|\leq  Y^{-1+\eta}/\|\nabla f(\x)\| }
\omega(Y\theta \nabla f(\x)) \d \theta \\
=~& \frac{1}{Y\|\nabla f(\x)\|} + O \left(\frac{\1(\x) }{Y\|\nabla f(\x)\|}\right) +  O_N(Y^{-N}).\end{align*}
Putting everything together yields the statement of the lemma.
\end{proof}

It is now time to return to our expression  \eqref{eq:MeR} for $M_\ve(R)$. 
First, sticking with the notation $X=R/k$ and $Y=R^{1-\ve}$,  we deduce from 
Lemma \ref{lem:major'} that 
$$
\sum_{\substack{
\x\in \ZZ^n\\
f(\x)=0\\\|\x\|\sim X}}
\mathsf{d}(k\x)
\int_{{\mathfrak{M}_\eta}(X,Y)} S(\beta) \d\beta
=
Y^{n-1}\sum_{\substack{
\x\in \ZZ^n\\
f(\x)=0\\\|\x\|\sim X}}
\mathsf{d}(k\x)
\cdot
 \frac{\gcd(\nabla f(\x))}{\|\nabla f(\x)\|}
\left\{
1+E(\x)
 \right\},
$$
where
$$
E(\x)=O\left( \1(\x)\right)
+
O_N(X^{d-1}Y^{-N})
$$
for any $N\geq 0$. 
If $\1(\x)=1$ then $\gcd(\nabla f(\x))>C_f^{-2}Y^{1-\eta}$. Furthermore, if $\gcd(x_1,\dots,x_n)=\ell$ 
then  \eqref{eq:elimination} yields
 \begin{align*}
 \gcd(\nabla f(\x))=
 \gcd( \ell^{d-1} \nabla f( \x/\ell) )
 &=
  \ell^{d-1} \gcd(\nabla 
f(\x/\ell)) \\
&\ll  \ell^{d-1} ,
\end{align*}
whence in fact $\ell \gg Y^{(1-\eta)/(d-1)}$.
Moreover, we have
$$
\mathsf{d}(k\x)
\cdot
 \frac{\gcd(\nabla f(\x))}{\|\nabla f(\x)\|}\leq k^{d-1}\gcd(\nabla f(\x)) \ll k^{d-1} \ell^{d-1} .
 $$

Assume now that $n>2^{d}(d-1)$. Then it follows from \eqref{eq:birch} that 
\begin{align*}
\#\left\{\x\in \ZZ^n: \begin{array}{l}
f(\x)=0,~\|\x\|\leq X\\ 
\gcd(x_1,\dots,x_n)=\ell
\end{array}
\right\} 
&\ll
 N_V( X/\ell)
 \ll \frac{X^{n-d}}{\ell^{n-d}}.
\end{align*}
Thus
$$
\sum_{\substack{
\x\in \ZZ^n\\
f(\x)=0\\\|\x\|\sim X}}
\mathsf{d}(k\x)
\cdot
 \frac{\gcd(\nabla f(\x))}{\|\nabla f(\x)\|}
E(\x)\ll_N  k^{d-1}\left(
X^{n-d}Y^{-\delta}+
X^{n-1}Y^{-N}\right)
$$
for any $N\geq 0$, where
$$
\delta=\frac{(1-\eta)(n-2d)}{d-1}.
$$ Here the exponent $\delta$ arises from summing the $\ell^{-(n-d) + (d-1)} = \ell^{-(n-2d+1)}$ savings over $\ell \gg Y^{(1-\eta)/(d-1)}$.
Reintroducing the sum over $k$, we now see that 
the overall contribution to 
$M_\ve(R)$ from the set of major arcs $\mathfrak{M}_{\eta,k}=\mathfrak{M}_\eta(R/k,R^{1-\ve})$  is 
\begin{align*}
\sum_{k\leq R}\mu(k)
\hspace{-0.2cm}
\sum_{\substack{
\x\in \ZZ^n\\
f(\x)=0\\\|\x\|\sim R/k}}
\mathsf{d}(k\x)
\int_{\mathfrak{M}_{\eta,k}} 
\hspace{-0.2cm}
S(\beta) \d\beta
&=
R^{(1-\ve)(n-1)}
\hspace{-0.2cm}
\sum_{\substack{
\x\in \ZZp^n\\
f(\x)=0\\\|\x\|\sim R}}1
+O\left(R^{n-d +(1-\ve)(n-1-\delta)}\right),
\end{align*}
on taking $N$ sufficiently large.

Putting 
 $\mathfrak{m}_{\eta,k}=[0,1)\setminus \mathfrak{M}_{\eta,k}$ and 
bringing everything together in 
\eqref{eq:E_eps}, it now follows that 
\begin{align*}
E_{V,\ve}^*(R) 
&\ll R^{n-d-(1-\ve)\delta}+
\left|
\sum_{k\leq R}\mu(k)
\sum_{\substack{
\x\in \ZZ^n\\
f(\x)=0\\\|\x\|\sim R/k}}
\frac{\mathsf{d}(k\x)}{R^{(1-\ve)(n-1)}}
\int_{\mathfrak{m}_{\eta,k}} S(\beta) \d\beta\right|.
\end{align*}
We may detect the equation $f(\x)=0$ in the way most familiar to practitioners of the
Hardy--Littlewood circle method.  On 
doing so, we are led to the following result, which   summarises our discussion of the major arcs.

\begin{lemma}\label{lem:major-final}
Let $n>2^{d}(d-1)$.	
For any $k\in \NN$ let 
 $\mathfrak{m}_{\eta,k}=[0,1)\setminus \mathfrak{M}_{\eta,k}$, where
 $\mathfrak{M}_{\eta,k}=\mathfrak{M}_\eta(R/k,R^{1-\ve})$ is given by Definition \ref{def:major}.
Then there exists $\delta>0$ such that 
$$
E_{V,\ve}^*(R) \ll R^{n-d-\delta}+
\frac{1}{R^{(1-\ve)(n-1)}}
\sum_{k\leq \sqrt{R}}\mu^2(k)
\int_0^1\int_{\mathfrak{m}_{\eta,k}} |S(\alpha,\beta)| \d\alpha\d\beta,
$$
where if $\mathsf{d}(\x)$ is given by \eqref{eq:dx} then
$$
S(\alpha,\beta)=
\sum_{\substack{
\x\in \ZZ^n\\
\|\x\|\sim R/k}}
 \sum_{\y \in \ZZ^n} 
 \mathsf{d}(k\x)
 \omega(\y/R^{1-\ve})  e(\alpha f(\x)+\beta \y.\nabla f(\x)) .
$$
\end{lemma}

\begin{proof}
The only thing that requires comment is the truncation from $k\leq R$ to $k\leq \sqrt{R}$. But since $\mathsf{d}(k\x)\ll R^{d-1}$ the trivial bound yields
$$
\frac{|S(\alpha,\beta)|}{R^{(1-\ve)(n-1)}}
\ll 
\frac{R^{d-1}\cdot (R/k)^n\cdot R^{(1-\ve)n}}{R^{(1-\ve)(n-1)}}
\ll \frac{R^{n+d-\ve}}{k^n}.
$$
Hence the tail of the $k$-summation makes  a satisfactory contribution.
\end{proof}

\section{Treatment  of the minor arcs}

We begin with a technical result from the geometry of numbers, which generalises the ``shrinking lemma'' that is due to   Davenport \cite[Lemma 12.6]{dav}, and which one recovers by taking $P=Q$ in the following result. 

\begin{lemma}\label{new-geometry} 
Let $\gamma$ be a symmetric $n \times n$ matrix with entries in $\RR$. 
Let $P>0$, let $Q> 2$ and let $\theta \in (0,1]$.
Let $N_{\gamma,P,Q}$ be the number of ${\bf x} \in \ZZ^n$ such that $\| {\bf x} \| < P$ and $\max_{1\leq i\leq n}\langle  \gamma x_i \rangle < Q^{-1}$.  Then
\[ 
\frac{N_{\gamma,P,Q}}{N_{\gamma,\theta P,\theta^{-1} Q}} \ll \theta^{-n} 
\max\left\{\sqrt{\frac{P}{Q}},1\right\}^n , 
\] 
where the implied constant depends only on $n$.
\end{lemma}

\begin{proof}
We may assume that $P\geq 1$, since the left hand side is $1$ when $P<1$.
Define the matrix
\[ \Lambda_{P,Q} = \begin{pmatrix} P^{-1} I_n & 0 \\ Q \gamma & Q I_n \end{pmatrix} ,
\] 
so that 
\[ 
\Lambda_{P,Q}^{-t} = \begin{pmatrix} P I_n & 
-P
\gamma \\ 0 & Q^{-1} I_n \end{pmatrix} \quad 
\text{ and } \quad 
\frac{Q}{P}
\Lambda_{P,Q}^{-t} 
= 
\begin{pmatrix} 0 & I_n \\ - I_n & 0 \end{pmatrix} \Lambda_{P,Q}  \begin{pmatrix} 0 & I_n \\ - I_n & 0 \end{pmatrix} ^{-1} .\] 
Let $R_1\leq  \dots \leq R_{2n}$ denote the successive minima of the lattice corresponding to $\Lambda_{P,Q}$ and let 
 $R_1^*\leq  \dots \leq R_{2n}^*$ be the successive minima of the dual lattice corresponding to $(Q/P)\Lambda_{P,Q}^{-t}$. Then  \eqref{eq:*} implies that  
$
R_i^*\asymp (Q/P)/R_{2n-i+1},
$
for $1\leq i\leq 2n$.
Since the lattices are equal up to left and right multiplication by a matrix in $\mathrm{GL}_{2n}(\ZZ)$, 
we must have  
$$
R_i\asymp \frac{Q/P}{ R_{2n+1-i}}
$$ 
for all $1\leq i\leq 2n$. Taking $i=n+1$ we deduce that 
$
\sqrt{Q/P} \ll R_{n+1}.
$ 

Since $Q>2$, 
the quantity  $N_{\gamma,P,Q}$ is equal to the number of vectors in the lattice corresponding to $\Lambda_{P,Q}$ whose first $n$ entries form a vector of Euclidean norm $<1$ and whose last $n$ entries are individually $<1$.  Thus it is bounded below by the number of vectors with Euclidean norm $<1$, and bounded above by the corresponding number with Euclidean norm $<\sqrt{n+1}$. 
On the other hand,  $N_{\gamma,\theta P,\theta^{-1}Q}$ is bounded below by the number 
of vectors in the lattice corresponding to 
$\Lambda_{P,Q}$ 
with  norm $<\theta$ and above by the corresponding  number with norm $<\theta \sqrt{2n+1}$.
It therefore follows from Davenport
\cite[Lemma~12.4]{dav} that 
$$
N_{\gamma,P,Q} \asymp \prod_{i=1}^{2n} \max\{1, R_{i}^{-1}\}\quad 
\text{ and 
}\quad N_{\gamma,\theta P,\theta^{-1}Q}\asymp \prod_{i=1}^{2n}  \max\left\{1, \theta R_i^{-1}\right\},
$$ 
where the implied constants depend only on $n$.
Dividing term by term, we see that each $i$ contributes at most $\theta^{-1}$ and each $i\geq n+1$ contributes at most $\max\{\sqrt{P/Q},1\}$. Thus the total contribution is at most 
$\theta^{-n} 
\max\left\{\sqrt{P/Q},1\right\}^n$, as claimed in the statement of the lemma.  
\end{proof}

The second technical result required is a simple Diophantine approximation result 
due to Heath-Brown \cite[Lemma 2.3]{14}.

\begin{lemma}\label{23-approx}
Let $M,R>0$.
Let $m\in \ZZ$ such that $|m|\leq M$ and let 
$\alpha=a/q+z$, with coprime integers $a,q$  and  $z\in \RR$, 
such that $\langle \alpha m\rangle <R^{-1}$.
Assume that 
$$
|z|\leq (2qM)^{-1}, \quad q\leq R/2 \quad \text{ and } \quad 
q>\min\{M, (|z|R)^{-1}\}.
$$
Then $m=0$.  
\end{lemma}

The statement of this result requires the assumption that $a$ and $q$ are coprime, which isn't formally stated in \cite[Lemma 2.3]{14} but is implicit in the proof.

We now have the tools in place to study our exponential sum on the minor arcs.
Let us set
\begin{equation}\label{eq:XY}
X=R/k \quad \text{ and } \quad Y=R^{1-\ve},
\end{equation}
as previously. 
We want to study
$$
S(\alpha,\beta)=
\sum_{\substack{
\x\in \ZZ^n\\
\|\x\|\sim X}}
 \sum_{\y \in \ZZ^n} 
 \mathsf{d}(k\x)
 \omega(\y/Y)  e(g(\x,\y)),
$$
for $(\alpha, \beta)\in [0,1)\times \mathfrak{m}_{\eta,k}$, 
where
$\mathsf{d}(\x)$ is given by \eqref{eq:dx} and 
 $$
 g(\x,\y)=\alpha f(\x)+\beta \y.\nabla f(\x).
 $$
 Let us write $\widetilde\Delta_f=\Delta_f/\gcd(k^{d-1},\Delta_f)$.
 Then we have 
 \[  \mathsf{d}(k\x) = \frac{ \| \nabla f(k\x) \| }{ \gcd(\nabla f(k \x), \Delta_f ) } = \frac{k^{d-1}\|\nabla f(\x)\|}{\gcd(k^{d-1},\Delta_f)\gcd(\nabla f(\x),\widetilde \Delta_f)}, \] 
 whence
\begin{equation}\label{eq:rain}
S(\alpha,\beta)=
\frac{k^{d-1}}{\gcd(k^{d-1},\Delta_f)} 
\sum_{\substack{
\x\in \ZZ^n\\
\|\x\|\sim X}}
 \sum_{\y \in \ZZ^n} 
\frac{\|\nabla f(\x)\|}{\gcd(\nabla f(\x),\widetilde \Delta_f)}
 \omega(\y/Y)  e(g(\x,\y)).
\end{equation}

 In this section an important role will be played by the  multilinear forms
$$
m_j(\x^{(1)},\dots,\x^{({d-1})})
=
d! 
\sum_{j_1,\dots,j_{d-1}=1}^n c_{j_1,\dots,j_{d-1},j} 
x_{j_1}^{(1)}\dots x_{j_{d-1}}^{(d-1)},
$$
for $1\leq j\leq n$, where 
$c_{j_1,\dots,j_d} \in\ZZ$  are the
symmetric coefficients such that 
$$
f(\x)
=\sum_{j_1,\dots,j_d=1}^n c_{j_1,\dots,j_d} 
x_{j_1}\dots x_{j_d}.
$$
In what follows we shall write 
$\underline{\u}$ to denote the vector $(\u_1,\dots,\u_{d-1})$.
Since $f$ is non-singular, it follows from 
\cite[Lemmas 3.1 and 3.3]{birch} that 
\begin{equation}\label{eq:rogan}
\#\left\{
 \underline{\u}\in \ZZ^{(d-1)n}:
\begin{array}{l}
 \|\u_1\|,\dots, \|\u_{d-1}\|<U\\
 m_j(\u_1,\dots,\u_{d-1}) =0 ~\forall j\leq n
\end{array}
\right\}\ll 1+ U^{(d-2)n}
\end{equation}
for any $U> 0$.
Next, 
for given $P>0, Q> 2$ and $\tau\in \RR$, let  
 \begin{equation}\label{eq:PQ}
 \mathcal{M}(\tau;P,Q)=\#\left\{
 \underline{\u}\in \ZZ^{(d-1)n}:
\begin{array}{l}
 \|\u_1\|,\dots, \|\u_{d-1}\|<P\\
 \left\langle 
\tau m_j(\u_1,\dots,\u_{d-1}) 
\right\rangle<Q^{-1} ~\forall j\leq n
\end{array}
\right\}.
\end{equation}
It follows from $d-1$  applications of Lemma \ref{new-geometry} that
\begin{equation}\label{eq:shrink}
 \mathcal{M}(\tau;P,Q)\ll 
 \frac{\max\{\sqrt{P/Q}, 1\}^{(d-1)n}}{\theta^{(d-1)n}}
 \mathcal{M}(\tau;\theta P,\theta^{1-d}Q),
\end{equation}
for any $\theta\in (0,1]$. 

 Returning to the expression for $S(\alpha,\beta)$ in \eqref{eq:rain}, we start by removing the factor $\gcd(\nabla f(\x),\widetilde\Delta_f)$ via the observation that $\gcd(\nabla f(\x),\widetilde\Delta_f)$ depends only on $\x$ mod $\widetilde\Delta_f$. 
Letting $h= \widetilde\Delta_f$ for compactness of notation, 
 we break the sum into residue classes mod $h$, getting 
 \begin{equation}\label{eq:attic}
S(\alpha,\beta)=
\frac{k^{d-1}}{\gcd(k^{d-1},\Delta_f)} 
\sum_{\y\in \ZZ^n}  \omega(\y/Y) 
\sum_{\substack{\bxi\in (\ZZ/h\ZZ)^n
}} \frac{1}{\gcd(\nabla f(\bxi), h)}
T(\y),
\end{equation}
where
$$
T(\y)=
\sum_{\substack{
 \x\in \ZZ^n\\
\|\x\|\sim X\\
\x\equiv \bxi \bmod{h}
}}
\|\nabla f(\x)\|
e(g(\x,\y)).
$$

We may write
$$
T(\y)=\sum_{\x\in \ZZ^n} F(\x) e(G(\x)), 
$$
where
$$
F(\x)=\begin{cases}
\|(\nabla f)(\bxi+h\x)\| &\text{ if $\|\bxi+h\x\| \sim X$,}\\
0 & \text{ otherwise,}
\end{cases}
$$
and 
$$
 G(\x)=\alpha f(\bxi+h\x)+\beta \y.(\nabla f)(\bxi+h\x).
 $$
 Note that $G(\x)$ has degree $d$. 
 
 We shall estimate $T(\y)$ via Weyl differencing,
as in Birch \cite{birch}. 
  Let 
 \begin{align*}
 F_{\u_1}(\x)&=F(\x+\u_1)F(\x),\\
 F_{\u_1,\u_2}(\x)&=F(\x+\u_1+\u_2)F(\x+\u_1)F(\x+\u_2)F(\x),\\ 
&\hspace{0.25cm}\vdots
 \end{align*}
and 
 \begin{align*}
 G_{\u_1}(\x)&=G(\x+\u_1)-G(\x),\\
 G_{\u_1,\u_2}(\x)&=G(\x+\u_1+\u_2)-G(\x+\u_1)-G(\x+\u_2)+G(\x),\\ 
&\hspace{0.25cm}\vdots
 \end{align*}
Then, for any $r\in \{1,\dots,d-1\}$, we have 
\begin{equation}\label{eq:weyl}
\left| \frac{T(\y)}{X^{n}}\right|^{2^r} \ll \frac{1}{X^{rn} } \sum_{\|\u_1\|<X} \dots
\sum_{\|\u_r\|<X}  \left|\frac{T_{\u_1,\dots,\u_r}(\y)}{X^n}\right|,
\end{equation}
where
$$
T_{\u_1,\dots,\u_r}(\y)=\sum_{\x\in \ZZ^n} 
F_{\u_1,\dots,\u_r}(\x) e(G_{\u_1,\dots,\u_r}(\x)).
$$
We shall produce two estimates for $T(\y)$. 
In the first we  take $r=d-1$, which  eliminates the effect of the lower degree term $\beta \y.(\nabla f)(\bxi+h\x)$ and leads  to a family of linear exponential sums that depend on the Diophantine approximation properties of $\alpha$ alone.
Alternatively, we take $r=d-2$. After  a further application of Cauchy--Schwarz, one  brings the $\y$-sum inside, thereby bringing the Diophantine properties of $\beta$ into play.

By Dirichlet's approximation theorem  there exist $a,q\in \ZZ$ and $\psi\in \RR$ such that 
$$
\alpha=\frac{a}{q}+\psi,
$$
with 
\begin{equation}\label{eq:dirichlet1}
\gcd(a,q)=1, \quad
0\leq a<q\leq X^{d/2} \quad \text{ and } \quad |\psi|\leq \frac{1}{qX^{d/2}}.
\end{equation}
The following is our first bound for $S(\alpha,\beta)$ and only involves the Diophantine approximation properties of $\alpha$.

\begin{lemma}\label{lem:5.1}
Assume that $\alpha=a/q+\psi$  is such that \eqref{eq:dirichlet1} holds and put
$$
D=\frac{n}{2^{d-1}(d-1)}.
$$
Then 
$$
S(\alpha,\beta)\ll \frac{
 k^{d-1} X^{n+d-1} Y^{n}(\log X)^n}{ q^D}
\min \left\{ 1, \frac{1}{|\psi| X^d}\right\}^{D}.
$$
\end{lemma}

\begin{proof}
Taking $r=d-1$ in \eqref{eq:weyl}, we first note that 
$$
G_{\u_1,\dots,\u_{d-1}}(\x)=\alpha h^{d}\sum_{j=1}^n x_j m_j(\u_1,\dots,\u_{d-1}) +H(\u_1,\dots,\u_{d-1}),
$$
for some polynomial $H(\u_1,\dots,\u_{d-1})$ that doesn't depend on $\x$. 
It follows that 
$$
\sum_{\substack{\x\in \ZZ^n\\
-cX\leq x_j\leq t_j}}e(G_{\u_1,\dots,\u_r}(\x))\ll \prod_{j=1}^n \min\left\{X, \langle
\alpha h^{d} m_j(\u_1,\dots,\u_{d-1}) \rangle^{-1}
 \right\},
$$
for any $t_1,\dots,t_n\ll X$ and any absolute constant $c>0$.
Exploiting  \eqref{eq:grad},  it readily follows from multi-dimensional partial summation that 
$$
T_{\u_1,\dots,\u_r}(\y)
\ll X^{2^{d-1}(d-1)}
\prod_{j=1}^n \min\left\{X, \langle
\alpha h^{d} m_j(\u_1,\dots,\u_{d-1}) \rangle^{-1}
 \right\}.
$$
In the standard way  (cf.\ the proof of \cite[Lemma 13.2]{dav})
one finds that 
$$
\left| \frac{T(\y)}{X^{n}}\right|^{2^{d-1}} \ll X^{(d-1)(2^{d-1}-n)}(\log X)^n
\mathcal{M}(\alpha h^{d}; X,X),
$$
in the notation of \eqref{eq:PQ}. Applying 
\eqref{eq:shrink} we obtain
$$
\frac{\mathcal{M}(\alpha h^{d}; X,X)}{X^{(d-1)n}}
\ll  
\frac{\mathcal{M}(\alpha h^{d}; \theta X,\theta^{1-d}X)}{(\theta X)^{(d-1)n}},
$$
for any $\theta\in (0,1]$. By choosing 
$\theta$ to satisfy
$$
\theta^{d-1}\asymp  \min\left\{
1, \frac{1}{|q\psi| X^{d-1}}, \frac{X}{q}, \max\left\{
\frac{q}{X^{d-1}}, |q\psi|X
\right\}\right\},
$$
for appropriate implied constants depending on $f$, 
we can  make 
Lemma \ref{23-approx} applicable. We then deduce from  \eqref{eq:rogan}
that
\begin{align*}
\frac{\mathcal{M}(\alpha h^{d}; X,X)}{X^{(d-1)n}}
&\ll
\max\left\{
\frac{1}{X^{d-1}}, q |\psi| , \frac{q}{X^d}, \min \left\{ \frac{1}{q}, \frac{1}{q|\psi| X^d}\right\}
\right\}^{n/(d-1)}\\
&\ll  \frac{1}{q^{n/(d-1)}}
\min \left\{ 1, \frac{1}{|\psi| X^d}\right\}^{n/(d-1)},
\end{align*}
since  \eqref{eq:dirichlet1} holds. 
It follows that 
$$
T(\y)
\ll \frac{X^{n+d-1}(\log X)^n}{q^D}
\min \left\{ 1, \frac{1}{|\psi| X^d}\right\}^{D},
$$
with $D$ as in the statement of  the lemma.
Substituting this into \eqref{eq:attic}, 
we conclude the proof of the lemma by summing trivially over $\y$ and the finitely many possible values of $\bxi$.
\end{proof}

We now turn to our alternative estimate for $S(\alpha,\beta)$, which is obtained by exploiting the Diophantine approximation properties of $\beta.$ 
By Dirichlet's approximation theorem  there exist $b,r\in \ZZ$ and $\rho\in \RR$ such that 
$$
\beta=\frac{b}{r}+\rho,
$$
with 
\begin{equation}\label{eq:dirichlet2}
\gcd(b,r)=1, \quad
0\leq b<r\leq \sqrt{X^{d-1}Y} \quad \text{ and } \quad |\rho|\leq \frac{1}{r\sqrt{X^{d-1}Y}}.
\end{equation}
We shall prove the following result, which operates under the assumption that $X$ and $Y$ are not too lopsided.

\begin{lemma}\label{lem:5.2}
Assume that $\beta=b/r+\rho$  is such that \eqref{eq:dirichlet2} holds and put 
$$
E=\frac{n}{2^{d-2}(d-1)}.
$$
Assume that $Y\leq X^{d-1}$.
Then 
$$
S(\alpha,\beta)\ll \frac{
 k^{d-1}X^{n+d-1} Y^{n}
\max\{X/Y, 1\}^{(d-1)n/2^{d-1}} 
 (\log X)^n}{ r^{E}}
\min \left\{ 1, \frac{1}{|\rho| X^{d-1}Y}\right\}^{E}.
$$
\end{lemma}

\begin{proof}
This time we 
begin through an application of H\"older's inequality in \eqref{eq:attic}.
Recalling that $h=O(1)$,
we deduce that 
\begin{align*}
\left|\frac{S(\alpha,\beta)}{k^{d-1}}\right|^{2^{d-2}}
&\ll
Y^{(2^{d-2}-1)n} 
\sum_{\|\y\|\ll Y}
\left(
\sum_{\substack{\bxi\in (\ZZ/h\ZZ)^n}}
\hspace{-0.2cm}
|T(\y)|\right)^{2^{d-2}}\\
&\ll
Y^{(2^{d-2}-1)n} 
\sum_{\substack{\bxi\in (\ZZ/h\ZZ)^n}}
U(\bxi),
\end{align*}
where
$$
U(\bxi)
=\sum_{\|\y\|\ll Y}
|T(\y)|^{2^{d-2}}.
$$
Taking $r=d-3$ in \eqref{eq:weyl}, it follows that 
\begin{align*}
U(\bxi)
&\ll \sum_{\|\y\|\ll Y}
\left|
 \frac{
 X^{2^{d-3}n} }{X^{(d-3)n} } \sum_{\|\u_1\|<X} \dots
\sum_{\|\u_{d-3}\|<X}  \left|\frac{T_{\u_1,\dots,\u_{d-3}}(\y)}{X^n}\right|\right|^2\\
&\ll
 \frac{
 X^{2^{d-2}n} }{X^{(d-1)n} }
 \sum_{\|\y\|\ll Y}
 \sum_{\|\u_1\|<X} \dots
\sum_{\|\u_{d-3}\|<X}  \left|T_{\u_1,\dots,\u_{d-3}}(\y)\right|^2.
\end{align*}
At this point we carry out a further differencing operation to conclude that
$$
|T_{\u_1,\dots,\u_{d-3}}(\y)|^2=\sum_{\u_{d-2}\in \ZZ^n}\sum_{\u_{d-1}\in \ZZ^n} 
F_{\u_1,\dots,\u_{d-2}}(\u_{d-1}) e(G_{\u_1,\dots,\u_{d-2}}(\u_{d-1})).
$$
There exists a polynomial  $H\in \ZZ[\u_1,\dots,\u_{d-1}]$ that doesn't depend on $\y$ and polynomials $r_1,\dots,r_n\in \ZZ[\u_1,\dots,\u_{d-2}]$ that don't depend on $\u_{d-1}$ such that 
\begin{align*}
G_{\u_1,\dots,\u_{d-2}}(\u_{d-1})=~& \beta h^{d-1}\sum_{j=1}^n y_j \big(m_j(\u_1,\dots,\u_{d-1})+r_j(\u_1,\dots,\u_{d-2})\big)\\ 
\quad&+H(\u_1,\dots,\u_{d-1}),
\end{align*}
where  $r_j$ may depend on $\bxi$.
We may now interchange the order of summation and  execute the sum over $\y$,
noting that $F_{\u_1,\dots,\u_{d-2}}(\u_{d-1})\ll X^{2^{d-2}(d-1)}$ for $\|\u_i\|< X$.
Hence, in the usual way,  
we  conclude that 
\begin{align*}
U(\bxi)
\ll~&
 \frac{
 X^{2^{d-2}(n+d-1)} }{X^{(d-1)n} }\sum_{\|\u_1\|<X} \dots
\sum_{\|\u_{d-1}\|<X} \\
 &
 \hspace{-0.3cm}
 \times
 \prod_{j=1}^n \min \left\{
Y, \left\langle 
\beta h^{d-1} \left(m_j(\u_1,\dots,\u_{d-1}) +r_j(\u_1,\dots,\u_{d-2})\right)
\right\rangle^{-1}
\right\}\\
\ll~&
 \frac{ X^{2^{d-2}(n+d-1)} Y^n(\log Y)^n}{X^{(d-1)n}}
\widetilde{\mathcal{M}}
\end{align*}
where 
$\widetilde{\mathcal{M}}
$
denotes  the number of 
$ \underline{\u}\in \ZZ^{(d-1)n}$ for which 
$ \|\u_1\|,\dots, \|\u_{d-1}\|<X$ and 
$$
 \left\langle 
 \beta h^{d-1} \left(m_j(\u_1,\dots,\u_{d-1}) +r_j(\u_1,\dots,\u_{d-2})\right)
\right\rangle<Y^{-1} ,
$$
for $1\leq j\leq n$. Note that  $m_j$ is linear when viewed as a polynomial in $\u_{d-1}$. For fixed $\u_1,\dots,\u_{d-2}$, given a single $\u_{d-1}'$ satisfying the inequality, for all other solutions $\u_{d-1}$ we will have 
\begin{align*}  \langle 
 \beta h^{d-1} (m_j &(\u_1,\dots,\u_{d-1} - \u_{d-1}' ) )
\rangle\\  
&=     \left\langle 
 \beta h^{d-1} \left(m_j(\u_1,\dots,\u_{d-1}) +r_j- m_j(\u_1,\dots,\u_{d-1}') - r_j\right)\right\rangle \\ 
 &<2 Y^{-1},     
 \end{align*}  
where $r_j=r_j(\u_1,\dots,\u_{d-2})$.
Thus we may replace  $\u_{d-1}$ by $\u_{d-1}- \u_{d-1}'$ to remove the constant term $r_j(\u_1,\dots,\u_{d-2})$. Doing so 
leads to the conclusion that 
$\widetilde{\mathcal{M}}
$
is at most the  number of vectors
$ (\u_1,\dots,\u_{d-2}, \u_{d-1} )\in  \ZZ^{(d-1)n}$ 
for which 
$ \|\u_1\|,\dots, \|\u_{d-2}\|<X$ and 
$\|\u_{d-1}\|<2 X$, with 
$$
 \left\langle 
 \beta h^{d-1} m_j(\u_1,\dots,\u_{d-1}) 
\right\rangle<2Y^{-1} ,
$$
for $1\leq j\leq n$. We conclude that 
$$
U(\bxi)
\ll
 \frac{ X^{2^{d-2}(n+d-1)} Y^n(\log Y)^n}{X^{(d-1)n}}
 \mathcal{M}(\beta h^{d-1};2X,Y/2),
$$
in the notation of \eqref{eq:PQ}.

Next, on appealing to \eqref{eq:shrink}, we deduce that 
$$
\mathcal{M}(\beta h^{d-1};2 X,Y/2) 
\ll 
 \frac{\max\{\sqrt{X/Y}, 1\}^{(d-1)n}}{\theta^{(d-1)n}}
 \mathcal{M}(\beta h^{d-1};2\theta X,\theta^{1-d}Y/2), 
$$
for any $\theta\in (0,1]$.
Assume that $\beta=b/r+\rho$. 
If we choose 
   $\theta $ so that 
$$
\theta^{d-1}\asymp \max\left\{ \frac{1}{X^{d-1}}, \min\left\{
1, \frac{1}{|r\rho| X^{d-1}}, \frac{Y}{r}, \max\left\{
\frac{r}{X^{d-1}}, |r\rho|Y
\right\}
\right\}\right\},
$$
for appropriate  implied constants that  depend only on $f$, then we can make 
Lemma \ref{23-approx}
applicable. In the light of \eqref{eq:rogan}, this leads to the conclusion that
 $$
 \mathcal{M}(\beta h^{d-1};2X,Y/2) 
\ll 
 \frac{\max\{\sqrt{X/Y}, 1\}^{(d-1)n} (\theta X)^{(d-2)n}}{\theta^{(d-1)n}},
 $$
 since $\theta X\gg 1$.
Thus
\begin{align*}
U(\bxi)
&\ll
 \frac{
 X^{2^{d-2}(n+d-1)} Y^n \max\{\sqrt{X/Y}, 1\}^{(d-1)n}(\log Y)^n}{(\theta X)^{n} }\\
&\ll 
 X^{2^{d-2}(n+d-1)} Y^n \max\{\sqrt{X/Y}, 1\}^{(d-1)n}(\log Y)^n M^{n/(d-1)},
\end{align*}
where 
\begin{align*}
M
&\ll
\max\left\{
\frac{1}{X^{d-1}}, r |\rho| , \frac{r}{YX^{d-1}}, \min \left\{ \frac{1}{r }, \frac{1}{r|\rho| YX^{d-1}}\right\}\right\}.
\end{align*}
Assuming that
\eqref{eq:dirichlet2} holds and $Y\leq X^{d-1}$, it follows that 
$$
M
\ll  \frac{1}{r}
\min \left\{ 1, \frac{1}{|\rho| X^{d-1}Y}\right\},
$$
whence finally
\begin{align*}
U(\bxi)
\ll ~&
 X^{2^{d-2}(n+d-1)} Y^n \max\{\sqrt{X/Y}, 1\}^{(d-1)n}(\log Y)^n
\\
&\times  \frac{1}{r^{n/(d-1)}}
\min \left\{ 1, \frac{1}{|\rho| X^{d-1}Y}\right\}^{n/(d-1)}.
\end{align*}
We deduce by summing over the finitely many possible values of $\bxi$ that 
\begin{align*}
\left|\frac{S(\alpha,\beta)}{k^{d-1}}\right|^{2^{d-2}}
\ll~&
\frac{X^{2^{d-2}(n+d-1)}
Y^{2^{d-2}n} 
\max\{\sqrt{X/Y}, 1\}^{(d-1)n}
(\log Y)^{n}}{r^{n/(d-1)}}\\
& \times
\min \left\{ 1, \frac{1}{|\rho| X^{d-1}Y}\right\}^{n/(d-1)}.
\end{align*}
The  lemma follows since $\log Y\leq (d-1)\log X$.
\end{proof}

We now have everything in place to complete the estimation of 
$E_{V,\ve}^*(R)$ via Lemma \ref{lem:major-final}. 
For the moment we continue to adopt the notation \eqref{eq:XY} for $X$ and $Y$. 
Since $k\leq \sqrt{R}$ in 
Lemma \ref{lem:major-final} we may assume that $Y\leq X^{d-1}$ in Lemma~\ref{lem:5.2}.
Given $Q_i,t_i>0$, let
$\mathfrak{I}(Q_1,Q_2;t_1,t_2)$ denote the overall contribution to the integral 
$$
\int_0^1\int_{\mathfrak{m}_{\eta,k}} |S(\alpha,\beta)| \d\alpha\d\beta
$$
from 
$\alpha=a/q+\psi$ and $\beta=b/r+\rho$ such that
$$
q\sim Q_1, \quad r\sim Q_2 \quad \text{ and } \quad  |\psi|\sim t_1,  \quad |\rho|\sim t_2
.$$
Then it follows that from Lemmas \ref{lem:5.1} and \ref{lem:5.2} that 
\begin{equation}\label{eq:4pm}\begin{split}
\mathfrak{I}(Q_1,Q_2;t_1,t_2)\ll~& k^{d-1}X^{n+d-1} Y^{n}
\max\{X/Y, 1\}^{(d-1)n/2^{d-1}} 
 Q_1^2Q_2^2t_1t_2\\
& \times (\log X)^n\min \left\{ 
 \frac{1}{Q_1^D}, \frac{1}{(Q_1t_1 X^{d})^D}
,\frac{1}{Q_2^E}, \frac{1}{(Q_2t_2X^{d-1}Y)^E}\right\}.
\end{split}\end{equation}
By invoking Dirichlet's approximation theorem twice,  as in 
\eqref{eq:dirichlet1} and \eqref{eq:dirichlet2}, we see that 
we are only interested in $Q_i,t_i>0$ such that 
$$
Q_1\ll X^{d/2}, \quad Q_1t_1\ll \frac{1}{X^{d/2}} \quad \text{ and }
\quad 
Q_2\ll \sqrt{X^{d-1}Y}, \quad Q_2t_2\ll \frac{1}{\sqrt{X^{d-1}Y}} .
$$
Furthermore, since $\beta$ belongs to the minor arcs $\mathfrak{m}_{\eta,k}$ it follows from 
Definition~\ref{def:major} that 
$\mathfrak{I}(Q_1,Q_2;t_1,t_2)=0$ unless 
$$
\max\left\{Q_2,Q_2t_2 X^{d-1}Y\right\}\gg Y^{1-\eta}.
$$
Since there are $O((\log XY)^4)$ possible dyadic values for $Q_i,t_i$ that can contribute, we get an estimate for the minor arc integral by taking a maximum of \eqref{eq:4pm} over all $Q_i,t_i$ satisfying these inequalities.

Taking $\min\{A,B\}\leq A^{2/D}B^{1-2/D}$,
with 
$$
A= \frac{1}{\max\{Q_1, Q_1t_1 X^{d}\}^D} \quad \text{ and }\quad
B=\frac{1}{\max\{Q_2, Q_2t_2X^{d-1}Y\}^E},
$$
and then taking $\max\{1,t_1X^{d}\}^2\geq t_1X^d$, we deduce from \eqref{eq:4pm} that 
\begin{align*}
\mathfrak{I}(Q_1,Q_2;t_1,t_2)\ll~& k^{d-1}X^{n-1} Y^{n}\max\{X/Y, 1\}^{(d-1)n/2^{d-1}} 
(\log X)^n \\
&\times\frac{Q_2^2t_2}{
 \max \left\{ 
Q_2,Q_2t_2X^{d-1}Y \right\}^{E(1-2/D)}}.
\end{align*}
But $2E/D=4$ and 
$
Q_2^2t_2 X^{d-1}Y\leq  \max \left\{ 
Q_2,Q_2t_2X^{d-1}Y \right\}^2.
$
Hence 
\begin{align*}
\mathfrak{I}(Q_1,Q_2;t_1,t_2)
&\ll \frac{k^{d-1}X^{n-d} Y^{n-1}
\max\{X/Y, 1\}^{(d-1)n/2^{d-1}} 
(\log X)^n}{
 \max \left\{ 
Q_2,Q_2t_2X^{d-1}Y \right\}^{E-6}}\\
&\ll
\frac{k^{d-1}X^{n-d} Y^{n-1}\max\{X/Y, 1\}^{(d-1)n/2^{d-1}} 
(\log X)^n}{Y^{(1-\eta)(E-6)}}.
\end{align*}
Note that the exponent of $Y$ in the denominator is strictly positive precisely when $n>3(d-1)2^{d-1}$.
Recalling that $X$ and $Y$ are given by \eqref{eq:XY} we insert this argument into Lemma 
\ref{lem:major-final} to deduce that 
$$
E_{V,\ve}^*(R)\ll R^{n-d-\delta}
$$
for some $\delta>0$, provided that $\ve$ is sufficiently small in terms of $d$ and $n$.  This completes the proof of Theorem \ref{t:P}, on summing over dyadic intervals in \eqref{eq:B-R}.

We can get an explicit value of the constant $c_{d,n}$ as follows. 
Since $Y \geq  X^{ 1-\ve} $ in \eqref{eq:XY},  we see that 
\[ \frac{X^{n-d} Y^{n-1}\max\{X/Y, 1\}^{(d-1)n/2^{d-1}} 
(\log X)^n}{Y^{(1-\eta)(E-6)}} \leq \frac{X^{n-d} Y^{n-1} X^{ \ve (d-1)n/2^{d-1}} 
(\log X)^n}{X^{(1-\ve) (1-\eta)(E-6)}} \] gives a power saving 
as soon as $(1-\ve) (1-\eta)(E-6) > \ve (d-1)n/ 2^{d-1}.$
Recalling that  $E = n / (2^{d-2} (d-1))$ and multiplying both sides by $2^{d-1} (d-1)$, this condition becomes
 \[(1-\ve) (1- \eta) ( 2n - 3 (d-1) 2^d ) > \ve n (d-1)^2\] or 
 \[ \ve <  \frac{ 2n   - 3 (d-1)  2^d } {  n  (d-1)^2/ (1-\eta) +2n- 3 (d-1)  2^d}. \] 
 Thus  we may take $c_{d,n} = \frac{ 2n   - 3 (d-1)  2^d } {  n  (d^2-2d+3) - 3 (d-1)  2^d} $ in Theorem \ref{t:P} by letting $\eta$ converge to $0$.  Note that for fixed $d$ we have 
 $c_{d,n}\to \frac{2}{d^2-2d+3}$ as $n\to \infty$.

\end{document}